\documentclass[12pt]{amsart}
\usepackage{amscd,amssymb,amsthm,verbatim}
\newcommand{\const}{\operatorname{const.}}

\newcommand{\diam}{\operatorname{diam}}

\newcommand{\dvol}{\operatorname{dvol}}

\newcommand{\Id}{\operatorname{Id.}}

\newcommand{\Ker}{\operatorname{Ker}}

\newcommand{\R}{{\mathbb R}}

\newcommand{\Rm}{\operatorname{Rm}}

\newcommand{\SL}{\operatorname{SL}}
\newcommand{\SO}{\operatorname{SO}}
\newcommand{\Sol}{\operatorname{Sol}}

\newcommand{\SU}{\operatorname{SU}}

\newcommand{\Tr}{\operatorname{Tr}}

\newcommand{\U}{\operatorname{U}}

\newcommand{\vol}{\operatorname{vol}}
\newcommand{\Z}{{\mathbb Z}}

\numberwithin{equation}{section}
\theoremstyle{plain}
\newtheorem{definition}[equation]{Definition}

\newtheorem{lemma}[equation]{Lemma}
\newtheorem{theorem}[equation]{Theorem}
\newtheorem{proposition}[equation]{Proposition}

\theoremstyle{remark}

\newtheorem{remark}[equation]{Remark}
\newtheorem{example}[equation]{Example}
\usepackage{graphicx}
\input{epsf}
\usepackage{a4wide}

\begin{document}

\title[Kasner-like regions near crushing singularities]
      {Kasner-like regions near crushing singularities}

\author{John Lott}
\address{Department of Mathematics\\
University of California, Berkeley\\
Berkeley, CA  94720-3840\\
USA} \email{lott@berkeley.edu}

\thanks{Research partially supported by NSF grant
DMS-1810700}
\date{November 30, 2020}

\begin{abstract}
We consider vacuum spacetimes with a crushing singularity.
Under some scale-invariant curvature bounds,
we relate the existence of Kasner-like regions
to the asymptotics of spatial volume densities.
\end{abstract}

\maketitle


\section{Introduction}

This paper is about the initial geometry of an expanding vacuum spacetime.
In Subsection \ref{subsec1.1} of this introduction we give some background about the
problems addressed.  In Subsection \ref{subsec1.2} we describe some of the
techniques used. In Subsection \ref{subsec1.3} we give broad descriptions
of the results.
A reader who is interested in the main results can skip to
Subsection \ref{subsec1.4}.

\subsection{Scenario} \label{subsec1.1}

We are concerned with expanding spacetimes that satisfy the vacuum
Einstein equations.  The word ``expanding'' refers to spatial slices and
its meaning will be clarified. One can study the asymptotic geometry
of such a spacetime in either the future or past directions.
The problems that one studies in the two directions are rather
different.

In the future direction, one can ask, for example, whether the
geometry becomes asymptotically homogeneous; see
\cite{Fischer-Moncrief (2002),Lott (2018)} and references therein.
In the other direction,
Hawking's singularity theorem says that under fairly general assumptions,
an expanding spacetime has a past singularity, in the sense of
an incomplete timelike geodesic
\cite[Theorem 4 on p. 272]{Hawking-Ellis (1973)}.
One basic question, relevant for the strong cosmic censorship question,
is whether the curvature generically blows up; see
\cite{Isenberg (2015)} and references therein. 
We are concerned with a somewhat different question, namely
the nature of the geometry as one approaches the singularity.

As a remark about the physical relevance of considering vacuum spacetimes, under
some assumptions there are heuristic arguments to say that the matter content
is irrelevant for the past asymptotics
\cite[Chapter 4]{Belinski (2017)}. Hence to simplify things, we only
consider vacuum spacetimes with vanishing cosmological constant.

One special but relevant
class of vacuum spacetimes comes from the Milne spacetimes.  In four
dimensions, a Milne spacetime
is a flat spacetime that is a quotient of a forward
lightcone in Minkowski space. Equivalently, it is 
the Lorentzian cone over a three dimensional Riemannian manifold of
constant sectional curvature $-1$.  We define an $(n+1)$-dimensional
spacetime to be of
Milne type if it is a Lorentzian cone over a Riemannian $n$-manifold
whose Ricci tensor is $-(n-1)$ times the Riemannian metric; the
spacetime is then a vacuum spacetime. 

Another special but relevant
class of vacuum spacetimes consists of the Kasner spacetimes, which
exist in any dimension; see Example \ref{4.5}. They are usually
not flat but they
are self-similar, in the sense that they admit a timelike homothetic
Killing vector field. More generally, one has the
spatially homogeneous Bianchi spacetimes.  Among these, the so-called Mixmaster
spacetimes
\cite[Chapter 30.7]{Misner-Thorne-Wheeler (1973)}
have a left-$\widetilde{\SL(2, \R)}$ spatial invariance
(Bianchi VIII)
or a left-$\SU(2)$ spatial invariance (Bianchi IX).

The strongest results about geometric asymptotics, in the future or past,
assume some continuous symmetries.
Going beyond this, one has the
BKL conjectures for the asymptotics of 
a generic vacuum spacetime with an initial singularity,
as one approaches the singularity
\cite{BKL (1970),Belinski (2017)}. The conjectures are loosely formulated
but contain the following points:
\begin{enumerate}
\item The evolution at different spatial points asymptotically decouples.
\item For a given spatial point, the asymptotic
  evolution is governed by the ODE
  of a homogeneous vacuum spacetime of Bianchi type VIII or IX.
    \end{enumerate}

Generic Bianchi VIII and Bianchi IX vacuum spacetimes have been extensively
studied, with rigorous results in
\cite{Beguin (2010),Brehm (2016),Dutilleul (2019),Liebscher-Harterich-Webster-Georgi (2011),Ringstrom (2001)}.  Approaching the singularity, they have time regions of
Kasner-like geometry, with jumps from one Kasner-like geometry to
another.  The jumps are along Bianchi II trajectories and
occur chaotically, when viewed on the right time scale.

There does not seem to be strong evidence either way regarding the BKL
conjectures, although numerics indicate some positive aspects.  In this
paper we focus on the Kasner-like regions.  Rather than considering
generic vacuum spacetimes, we look for conditions that ensure the
existence of Kasner-like regions, and conditions that rule them out.

\subsection{Techniques} \label{subsec1.2}

Our approach to analyzing the future or past behavior of vacuum spacetimes
consists of three features.
\begin{itemize}
\item A way to rescale vacuum
  spacetimes, that allows one to consider blowdown or blowup limits.
  \item A monotonicity
    result.  
\item A way to take a convergent subsequence of a sequence of
  vacuum spacetimes.
\end{itemize}
In combination, one can use these features to prove results about
geometric asymptotics of a spacetime
by contradiction.  One identifies a class of
putative target spacetimes and assumes that
a sequence of (blowdown or blowup) rescalings of the given
vacuum spacetime
do not approach the target class.
One takes a convergent subsequence of the rescalings and uses the
monotonicity result to show that the limit does in fact
lie in the target class,
thereby obtaining a contradiction. 
This general approach
is used to study many geometric flows, such as the Ricci flow.
For the future behavior of vacuum spacetimes, it was
pioneered by Anderson \cite{Anderson (2001)}.

In order to make progress, we make a standard assumption
that the vacuum spacetime in question
has a foliation by
constant mean curvature (CMC) spatial hypersurfaces.
We assume that the spacetime is expanding, meaning that the
spatial slices have negative mean curvatures that increase toward the future.
Examples in which CMC foliations are known to occur come from 
crushing singularities, meaning that in the past
there is a
sequence of compact Cauchy surfaces
whose mean curvatures approach $- \infty$ uniformly
\cite{Eardley-Smarr (1979), Marsden-Tipler (1980)}.
If there is a crushing singularity then
there is a CMC foliation in the past by
compact hypersurfaces, whose mean
curvatures $H$ approach $- \infty$ \cite{Gerhardt (1983)}.

If the spacetime has dimension $n+1$,
define the Hubble time by $t = - \: \frac{n}{H}$. Let $X$ denote the spatial
manifold, with Riemannian metric $h(t)$.
Given a parameter $s > 0$, there is a natural way to rescale the
CMC vacuum spacetime to produce another one; see (\ref{old1.36}).

If $X$ is compact then
Fischer and Moncrief
proved the remarkable monotonicity statement
that $t^{-n} \vol(X, h(t))$ is nonincreasing in $t$
\cite{Fischer-Moncrief (2002)}; a closely related
monotonic quantity was considered by Anderson \cite{Anderson (2001)}.
More precisely,
\begin{equation}
  \frac{d}{dt} \left( t^{-n} \vol(X, h(t)) \right) =
  - \: t^{1-n} \int_X L \left| K^0 \right|^2  \dvol_h,
\end{equation}
where $L$ is the lapse function and
$K^0$ is the traceless second fundamental form of the hypersurface.
One shows that 
$t^{-n} \vol(X, h(t))$ is constant in $t$ if and only if
the vacuum spacetime is a Milne spacetime.  Based on this,
Fischer and Moncrief suggested that most of a spacetime, in the sense
of volume, approaches a Milne spacetime in the future.  Actually,
this is true relative to a limiting spatial volume density
$\dvol_\infty = \lim_{t \rightarrow \infty} t^{-n} \dvol_{h(t)}$.
Related to this, the future stability of Milne spacetimes with
compact spatial slices was shown by Andersson and Moncrief in
\cite{Andersson-Moncrief (2003),Andersson-Moncrief (2011)}.

It is also possible that $\dvol_\infty$ vanishes, in which case the future
asymptotics have different models and were studied in \cite{Lott (2018)}.

Because $t^{-n} \vol(X, h(t))$ is nonincreasing in $t$, it is bounded above
for large $t$ and so it is suitable for studying future behavior.  On the
other hand, to study past behavior, one wants an expression that is
monotonic in the other direction.  In \cite{Lott (2020)} it was shown that
$t^{-1} \vol(X, h(t))$ is nondecreasing in $t$ provided that the spatial
scalar curvature $R$ is nonpositive.  
More precisely,
\begin{equation}
  \frac{d}{dt} \left( t^{-1} \vol(X, h(t)) \right) =
  - \frac{1}{n} \int_X L R \dvol_h,
\end{equation}
If $X$ is a compact $3$-manifold with contractible universal cover, for
example if it is diffeomorphic to a torus, and if $R \le 0$ then
$t^{-1} \vol(X, h(t))$ is constant in $t$ if and only if the vacuum
spacetime is a Kasner spacetime. Thus the expression 
$t^{-1} \vol(X, h(t))$ plays a parallel role to the Fischer-Moncrief
expression $t^{-n} \vol(X, h(t))$, in terms of characterizing model
geometries.

One can refine the monotonicity results to 
statements that are {\em pointwise} with respect to the spatial directions
\cite[(2.16)]{Lott (2018)},\cite[Corollary 5]{Lott (2020)}.
That is, one has monotonicity along any timelike curve that
intersects the spatial hypersurfaces orthogonally.  In this case, one does
not have to assume spatial compactness.

The next issue is about taking a convergent subsequence of a
sequence of spacetimes. In
the static case of Riemannian manifolds, there is a well developed theory
of (pre)compactness under uniform sectional curvature bounds, as
explained in \cite[Chapter 11]{Petersen (2016)}.
An extension to vacuum CMC spacetimes was developed by Anderson, who 
applied it to future asymptotics \cite{Anderson (2001)}.
In order to extract a blowdown limit, i.e. a subsequential limit of a
sequence of rescaled spacetimes,
the
curvature assumption was the scale-invariant condition that the
curvature is $O(t^{-2})$ as
$t \rightarrow \infty$.  We review this material in Subsection \ref{subsec2.2}.
In our case, to extract a blowup limit,
the curvature assumption becomes that the curvature is
$O(t^{-2})$ as $t \rightarrow 0$.  Following Ricci flow terminology, we
call this a type-I vacuum solution. 

\subsection{Objectives} \label{subsec1.3}

The goal of this paper is to give information about the geometry of
an expanding vacuum spacetime near an initial singularity, without any symmetry
assumptions. We assume the existence of a CMC foliation. We also assume
a type-I curvature bound.  We find that
the rate of decay of the spatial volume density leads to information about
the local geometry.  For example,
if the spatial volume density decays at the fastest possible rate as one
approaches the singularity,
namely $t^{n}$, then we show that the local geometry is asymptotically
of the Milne type.

Motivated by the BKL conjectures, we are also interested in characterizing
Kasner-type regions.  As an improvement to \cite{Lott (2020)}, rather than
assuming nonpositive spatial scalar curvature, we just assume that the
spatial scalar curvature is asymptotically nonpositive; this seems to often
be the case.  Under this
assumption, we show that if the spatial volume density
strictly decays at the
slower rate $t^1$ then the local geometry is asymptotically
Kasner-like as one approaches the singularity.

For the Mixmaster solutions that arise in the BKL conjectures, 
the decay of the spatial volume density is not strictly $t^1$.  In fact,
the decay is a bit faster, because of the Bianchi-II
transitions between the Kasner-type regions. In order to cover this
situation, we consider the case when the spatial volume density is
$O(t)$ as $t \rightarrow 0$,
but is not $O(t^{1+\beta})$ for any $\beta > 0$. (An example to have in mind
is $\frac{t}{\log \frac{1}{t}}$.) In this case we show that in a technical
sense, almost all of the time along a trajectory approaching the singularity
is spent in Kasner-like geometries.  

The overall theme is that under some reasonable {\it a priori} assumptions,
a condition about an object of low
regularity, the spatial volume density,
leads to conclusions about the local geometry as one approaches the
singularity.

\subsection{Main results} \label{subsec1.4}

We consider vacuum spacetimes that are diffeomorphic to $(0, t_0] \times X$,
where $X$ is an $n$-dimensional manifold, possibly noncompact.
After performing spatial diffeomorphisms, the
spacetime metric takes the form
$g = - L^2 dt^2 + h(t)$, where $L$ is the lapse function and
$h(\cdot)$ is a family of Riemannian
metrics on $X$.

We assume that each spatial slice has constant mean curvature, so we
have a CMC foliation. As mentioned before, this is the case in the neighborhood
of a crushing singularity.
We are interested in the expanding case, so
we assume that the mean curvature $H$ is
    monotonically increasing in $t$
    and takes all values in
    $(- \infty, H_0)$ for some $H_0 < 0$. We can then use the
    Hubble time given by $t = - \: \frac{n}{H}$.

    In terms of the time evolution, we can alternatively think of
    a vacuum spacetime with a spatial foliation
    as a flow, the Einstein flow.
    Let $|\Rm|_T$ be the pointwise curvature norm defined in (\ref{old1.47}).
    
\begin{definition} \label{1.1}
  A type-I Einstein flow is a CMC Einstein flow
  for which there is some $C < \infty$ so that
  $|\Rm|_T \le C t^{-2}$ for all $t \in (0, t_0]$.
\end{definition}
    
Given an Einstein flow and a parameter $s > 0$,
there is a rescaled Einstein flow ${\mathcal E}_s$, defined in
(\ref{old1.36}), with $s \rightarrow 0$ corresponding to a blowup limit
as one approaches the singularity at time zero.
The inequality in Definition \ref{1.1} is scale-invariant.

Our first result gives a situation where one can rule out the
existence of Kasner-like regions.
It describes the $t \rightarrow 0$ asymptotics at
a point $x \in X$ where the spatial volume density $\dvol_{h(t)}(x)$
has the scale invariant behavior $O(t^n)$. Since we localize around 
$x$, we use the notion of a pointed manifold, meaning a manifold with a
specified basepoint.
Let ${\mathcal M}$ denote the space of pointed Einstein flows
that correspond to Lorentzian cones over pointed Riemannian Einstein
$n$-manifolds with Einstein constant $-(n-1)$.
In what follows, $p$ will range over $[1, \infty)$ and $\alpha$ will
  range over $(0,1)$. The relevant notions of pointed convergence are
  defined in Subsection \ref{subsec2.2}.

\begin{theorem} \label{1.2}
  Suppose that ${\mathcal E}$ is a type-I CMC Einstein flow.
  Fix $x \in X$. 
  Suppose that 
  $\dvol_{h(t)}(x)$ is $O(t^n)$ as $t \rightarrow 0$.
  Then as $s \rightarrow 0$, the rescaled Einstein flows ${\mathcal E}_s$,
  pointed at $x$, approach ${\mathcal M}$ in the pointed weak
  $W^{2,p}$-topology and the pointed $C^{1,\alpha}$-topology.
\end{theorem}

Theorem \ref{1.2} describes the type-I Einstein
flows with the fastest possible volume
shrinkage as $t \rightarrow 0$. To deal with some Einstein flows with a slower
volume shrinkage, 
we introduce another class of special Einstein flows.
Let $K$ denote the second fundamental form, let
$R$ denote the spatial scalar curvature and
  let ${\mathcal K}$ be the collection of pointed expanding type-I CMC
  Einstein flows with
  $R = 0$, $L = \frac{1}{n}$ and $|K|^2 = H^2$, defined for $t \in
  (0, \infty)$.
  Examples of such Einstein flows are the $n$-dimensional
  Kasner flows (Example \ref{4.5}).
  If $X$ is a compact orientable $3$-manifold,
  and there is an aspherical component in the prime decomposition of
  $X$ (for example, if $X$ is diffeomorphic to $T^3$;
  see \cite[\S 6]{Scott (1983)} for definitions),
  then an element of ${\mathcal K}$
  is a Kasner flow in the sense
of Example \ref{4.5}
  \cite[Proposition 9]{Lott (2020)}.
  In particular, it is spatially flat.
  We do not know if this is true more generally, say if $X$ is
  a noncompact $3$-manifold and the spatial slices are complete.
  So we will think of elements of ${\mathcal K}$ as
  Kasner-like Einstein flows.

  \begin{definition}
    The CMC Einstein flow ${\mathcal E}$ has asymptotically nonpositive spatial
    scalar curvature if $\limsup_{t \rightarrow 0} \sup_{x \in X}
    t^2 R(t,x) \le 0$.
    \end{definition}

  The next theorem gives a sufficient condition to ensure that as
  $t \rightarrow 0$, the geometry becomes increasingly Kasner-like.
  
    \begin{theorem} \label{1.4}
      Let ${\mathcal E}$ be a type-I CMC Einstein flow with
      asymptotically nonpositive spatial scalar curvature.
      Suppose in addition that there is a nonnegative function
      $\widehat{R} : (0, t_0] \rightarrow \R$
with $\int_0^{t_0} t^2 \widehat{R}(t) \frac{dt}{t} < \infty$ so that
$R(t,x) \le \widehat{R}(t)$ for all $x \in X$ and $t \in (0, t_0]$.
Fix $x \in X$.  Suppose that
    there is some $c > 0$
    so that 
    $\frac{1}{t} \dvol_{h(t)}(x) \ge \frac{c}{t_0} \dvol_{h(t_0)}$
    for all $t \in (0, t_0]$.
    Then as $s \rightarrow 0$, the rescaled Einstein flows
    ${\mathcal E}_s$, pointed at $x$,
    approach ${\mathcal K}$ in the pointed weak
    $W^{2,p}$-topology
and the pointed $C^{1,\alpha}$-topology.
  \end{theorem}

    We note that even if there is pointwise convergence to a Kasner solution,
    the Kasner solution can definitely depend on $x$.
    Besides the obvious example of a Kasner solution (Example \ref{4.5}),
    examples of Theorem \ref{1.4} arise from Kantowski-Sachs solutions
    (Example \ref{4.8}), Taub-NUT solutions (Example \ref{4.10}),
    Bianchi VIII NUT solutions (Example \ref{4.11}) and, hypothetically,
    polarized Gowdy spacetimes (Example \ref{4.12}).

    From \cite[Lemma 1]{Lott (2020)}, if the Einstein
    flow has nonpositive spatial
    scalar curvature then the volume is $O(t)$ as $t \rightarrow 0$.
    In this case, Theorem \ref{1.4} was proven in
    \cite[Proposition 17]{Lott (2020)}. Also, if the
    Einstein flow has nonpositive
spatial scalar curvature then an integral version of Theorem \ref{1.4}
    was proven in \cite[Proposition 10]{Lott (2020)}, without
    any further curvature assumptions.
    
    Theorems \ref{1.2} and \ref{1.4} roughly describe two extremes of volume
    shrinkage as $t \rightarrow 0$, namely when the
    spatial volume density goes
    like $t^n$ or $t$.
    Some relevant spacetimes for the BKL conjectures are those
    whose volume densities may be $O(t)$, but are not
    $O \left( t^{1+\beta} \right)$ for any $\beta > 0$.
    The next theorem roughly says that as one approaches the singularity
    in such a spacetime,
    almost all of the time is spent in Kasner-like regions.
    
  \begin{theorem} \label{1.5}
    Suppose that ${\mathcal E}$ is a type-I CMC Einstein flow with asymptotically
    nonpositive spatial scalar curvature.  Fix $x \in X$.  Suppose that
    for each
    $\beta > 0$, 
    $\dvol_{h(t)}(x)$ fails to be $O(t^{1+\beta})$ as $t \rightarrow 0$. Given
    $\epsilon > 0$, let $S_\epsilon$ be the set of $\tau \ge 0$ so that
    ${\mathcal E}_{t_0 e^{- \tau}}$, pointed at $x$,
    is not $\epsilon$-close to ${\mathcal K}$ in the
    pointed $C^{1,\alpha}$-topology.
    Given $N \in \Z^+$, let $F(N)$ be the number of unit intervals
    $\{[k, k+1]\}_{k = 0}^{N-1}$ that have a nonempty intersection with
    $S_\epsilon$.
    Then
    \begin{equation}
      \liminf_{N \rightarrow \infty} \frac{F(N)}{N}
      = 0.
      \end{equation}
  \end{theorem}

  Examples of Theorem \ref{1.5} come from Mixmaster solutions
  of Bianchi type IX (Example \ref{2.55}) and Bianchi type VIII
  (Example \ref{2.79}). 
  One sees in the Mixmaster examples that the particular Kasner-like
  geometries that
  are approached can vary wildly as $\tau \rightarrow \infty$.
  
  Regarding the general
  applicability of Theorems \ref{1.2}, \ref{1.4} and \ref{1.5},
  we do not know an example of a
  vacuum spacetime, with a crushing singularity, that fails to be type-I with
  asymptotically nonpositive spatial scalar curvature. However, we have
  no reason to think that such examples do not exist.
  One feature of Theorems \ref{1.2}, \ref{1.4} and \ref{1.5} is that the volume
  assumptions are pointwise at $x \in X$.
  (More geometrically, we are looking at a trajectory that goes backward in time
  starting from $(t_0, x)$ and is normal to the spatial slices.)
  This is perhaps consistent with
  point (1) of the BKL conjectures, at least for
  Theorems \ref{1.4} and \ref{1.5}, where particle horizons can form.

  The proofs of Theorems \ref{1.2}, \ref{1.4} and \ref{1.5} are based on compactness
  arguments for Einstein flows, as initiated in
  \cite{Anderson (2001)}, and pointwise monotonicity formulas.
  For Theorem \ref{1.2}, the monotonicity formula is from
  \cite{Fischer-Moncrief (2002)}. For Theorems \ref{1.4} and \ref{1.5}, the
  formula is from \cite{Lott (2020)}.

  The structure of the paper is as follows. In Section \ref{sect2} we
  recall background material. In Section \ref{sect3} we
  prove Theorems \ref{1.2}, \ref{1.4} and \ref{1.5}, first under a 
  noncollapsing assumption. In this case, the spaces
  ${\mathcal M}$ and ${\mathcal K}$ consists of Einstein flows in the
  usual sense. In Subsection \ref{subsect3.4} we indicate how to remove the
  noncollapsing assumption, at the price of
  considering Einstein flows on
  a more general type of space. Section \ref{sect4} has examples of
  spacetimes that satisfy the hypotheses of
  Theorems \ref{1.2}, \ref{1.4} and \ref{1.5}, including Mixmaster examples.

  I thank the referees for their comments.
  
\section{CMC Einstein flows} \label{sect2}

In this section we recall some material about CMC Einstein flows,
rescalings and convergent subsequences.    

\subsection{CMC spacetimes} 

\begin{definition} \label{old1.1}
  Let $I$ be an interval in $\R$.
  An Einstein flow
  ${\mathcal E}$ on an $n$-dimensional manifold $X$ is given by a
  family of nonnegative functions $\{L(t)\}_{t \in I}$ on $X$,
  a family of Riemannian metrics $\{h(t)\}_{t \in I}$ on $X$, and a family
  of symmetric covariant $2$-tensor fields $\{K(t)\}_{t \in I}$ on $X$,
  so that if
  $H = h^{ij} K_{ij}$ and $K^0 = K - \frac{H}{n} h$ then
  the constraint equations
  \begin{equation} \label{old1.2}
    R - |K^0|^2 + \left( 1 - \frac{1}{n} \right) H^2 =  0
  \end{equation}
  and
  \begin{equation} \label{old1.3}
    \nabla_i K^i_{\: \: j} - \nabla_j H =  0,
  \end{equation}
  are satisfied, along with the evolution equations
  \begin{equation} \label{old1.4}
    \frac{\partial h_{ij}}{\partial t} = - 2 L K_{ij}
  \end{equation}
  and
  \begin{equation} \label{old1.5}
    \frac{\partial K_{ij}}{\partial t} =  L H K_{ij} - 2 L
    h^{kl} K_{ik} K_{lj} - L_{;ij} + L R_{ij}.
  \end{equation}
\end{definition}

For now, we will assume that all of the data
is smooth.
At the moment, $L$ is unconstrained; it will be determined by
the elliptic equation (\ref{old1.13}) below.
We will generally want $L(t)$ to be positive.

An Einstein flow gives rise to a Ricci-flat Lorentzian metric
\begin{equation} \label{old1.6}
  g = - L^2 dt^2 + h(t)
\end{equation}
on $I \times X$, for which the second fundamental form of the
time-$t$ slice is $K(t)$. On the other hand, given a
Lorentzian metric $g$ on a manifold with a proper time function $t$,
we can write it in the
form (\ref{old1.6}) by using
curves that meet the level sets orthogonally to give diffeomorphisms
between level sets and establish a product
structure.
Letting $K(t)$ be the second fundamental form of
the time-$t$ slice, the metric $g$ is Ricci-flat if and only if
$(L,h,K)$ is an Einstein flow.

\begin{definition} \label{old1.7}
  A CMC Einstein flow is an Einstein flow
  for which $H$ only depends on $t$.
\end{definition}

In the definition of a CMC Einstein flow, we do not assume that $X$ is compact.
We are interested in the expanding case, so
we assume that we have a CMC Einstein flow with $I = (0, t_0]$
    (or $I = (0, t_0)$), and that $H$ is
    monotonically increasing in $t$
    and takes all values in
    $(- \infty, H_0)$ for some $H_0 < 0$.
Important examples arise from crushing singularities as $t \rightarrow 0$,
in which case $X$ is compact by definition.
    
Returning to general expanding CMC Einstein flows,
equation (\ref{old1.5}) gives
\begin{align} \label{old1.13}
  \frac{\partial H}{\partial t} = & - \triangle_h L + LH^2
  + LR \\
  = & - \triangle_h L + L |K^0|^2 + \frac{1}{n} LH^2. \notag
\end{align}
Assuming bounded spatial curvature and bounded $L$ on compact time intervals,
the maximum principle gives
\begin{equation} \label{old1.14}
  L(t)
  \le \frac{n}{H^2} \frac{\partial H}{\partial t}.
\end{equation}

There is a pointwise identity
\begin{equation} \label{old1.16}
  \frac{\partial}{\partial t} \left( (-H)^n \dvol_h \right) =
  (-H)^{n+1} \left( L -  \frac{n}{H^2} \frac{\partial H}{\partial t}
  \right) \: \dvol_h.
\end{equation}
From (\ref{old1.14}),
it follows that $(-H)^n \dvol_{h(t)}$ is pointwise
monotonically nonincreasing
in $t$.

Also,
  \begin{equation} \label{2.20}
  \frac{\partial}{\partial t} \left( (-H) \dvol_h \right) =
  H^{2} \left( L -  \frac{1}{H^2} \frac{\partial H}{\partial t}
  \right) \: \dvol_h.
\end{equation}

\subsection{Rescaling limits} \label{subsec2.2} 

Let ${\mathcal E}$ be an Einstein flow. Let $g$ be the corresponding
Lorentzian metric.
For $s > 0$, the Lorentzian metric $s^{-2} g$ is isometric to
\begin{equation} \label{old1.35}
  g_s = - L^2(su) du^2 + s^{-2} h(su).
\end{equation}
Hence we put
\begin{align} \label{old1.36}
  & L_s(u) =  L(su),
  & h_s(u) = s^{-2} h(su), \: \: \:
  & K_{s,ij}(u) = s^{-1} K_{ij}(su),\\
  & H_s(u) = s H(su),
  & K^0_{s,ij}(u) = s^{-1} K^0_{ij}(su), \: \: \:
  & |K^0|^2_{s}(u) = s^2 K_{ij}(su), \notag \\
  & R_{s,ij}(u) = R_{ij}(su),
  & R_s(u) = s^2 R(su).
  & \notag
\end{align}
The variable $u$ will refer to the time parameter of a rescaled Einstein
flow, or a limit of such.
We write the rescaled Einstein flow as ${\mathcal E}_s$.

Put $e_0 = T = \frac{1}{L} \frac{\partial}{\partial t}$,
a unit timelike vector
that is normal to the level sets of $t$.
Let $\{e_i\}_{i=1}^n$ be an orthonormal basis for $e_0^\perp$. Put
\begin{equation} \label{old1.47}
  |\Rm|_T = \sqrt{\sum_{\alpha, \beta, \gamma, \delta = 0}^n
    R_{\alpha \beta \gamma \delta}^2}.
\end{equation}
Hereafter we assume that our CMC Einstein flows have complete spatial
slices, and that $|\Rm|_T$ is bounded on compact time intervals.

Let ${\mathcal E}^\infty = \left( L^\infty, h^{\infty}, K^{\infty}
\right) $ be a CMC Einstein flow on a pointed $n$-manifold
$\left( X^\infty, x^\infty \right)$,
defined on a time interval
$I^\infty$. 
Take $p \in [1, \infty)$ and $\alpha \in (0,1)$.

\begin{definition} \label{old1.48}
  The Einstein flow ${\mathcal E}^\infty$ is $W^{2,p}$-regular if
  $X^\infty$ is a $W^{3,p}$-manifold,
  $L^\infty$ and $h^\infty$ are locally $W^{2,p}$-regular in space and time, and
  $K^\infty$ is locally $W^{1,p}$-regular in space and time.
\end{definition}
Note that the equations of Definition \ref{old1.1} make sense in this
generality.

Let ${\mathcal E}^{(k)} = \{h^{(k)}, K^{(k)}, L^{(k)} \}_{k=1}^\infty$ be
smooth CMC
Einstein flows on pointed $n$-manifolds
$\{ \left( X^{(k)}, x^{(k)} \right) \}_{k=1}^\infty$,
defined on time intervals $I^{(k)}$.

\begin{definition} \label{old1.49}
  We say that $\lim_{k \rightarrow \infty} {\mathcal E}^{(k)} =
  {\mathcal E}^\infty$ in the pointed weak $W^{2,p}$-topology
  if
  \begin{itemize}
  \item Any compact interval $S \subset I^\infty$ is contained in
    $I^{(k)}$ for large $k$, and
  \item For any compact interval $S \subset I^\infty$ and
    any compact $n$-dimensional manifold-with-boundary $W^\infty \subset X^\infty$
    containing $x^\infty$, for large $k$
    there are pointed time-independent $W^{3,p}$-regular diffeomorphisms
    $\phi_{S,W^\infty,k} : W^\infty \rightarrow W^{(k)}$ (with
    $W^{(k)} \subset X^{(k)}$) so that
    \begin{itemize}
    \item $\lim_{k \rightarrow \infty}
      (\Id \times \phi_{S,W^\infty,k})^* L^{(k)} = L^\infty$ weakly in $W^{2,p}$ on $S \times W^\infty$,
    \item $\lim_{k \rightarrow \infty}
      (\Id \times \phi_{S,W^\infty,k})^* h^{(k)} = h^\infty$ weakly in $W^{2,p}$ on $S \times W^\infty$
      and
    \item $\lim_{k \rightarrow \infty}
      (\Id \times \phi_{S,W^\infty,k})^* K^{(k)} = K^\infty$ weakly in $W^{1,p}$ on $S \times W^\infty$.
    \end{itemize}
  \end{itemize}
\end{definition}

We define pointed (norm) $C^{1,\alpha}$-convergence similarly.

\begin{definition} \label{old1.50}
  Let ${\mathcal S}$ be a collection of pointed CMC
  Einstein flows defined on a time interval $I^\infty$.
  We say that a sequence
  $\{ {\mathcal E}^{(k)} \}_{k=1}^\infty$ of pointed CMC
  Einstein flows approaches ${\mathcal S}$
  as $k \rightarrow \infty$, in the pointed weak $W^{2,p}$-topology, if
  for any subsequence of $\{ {\mathcal E}^{(k)} \}_{k=1}^\infty$, there
  is a further subsequence
  that converges to an element of ${\mathcal S}$
  in the pointed weak $W^{2,p}$-topology.
\end{definition}

\begin{definition} \label{old1.51}
  Let ${\mathcal S}$ be a collection of pointed CMC
  Einstein flows defined on a time
  interval $I^\infty$. We say that a $1$-parameter family
  $\{ {\mathcal E}^{(s)} \}_{s \in (0,s_0]}$ of pointed CMC Einstein
    flows
    approaches
    ${\mathcal S}$ as $s \rightarrow 0$,
    in the pointed weak $W^{2,p}$-topology, if
    for any sequence $\{s_k\}_{k=1}^\infty$ in $(0, s_0]$ with
      $\lim_{k \rightarrow \infty} s_k = 0$, there is a subsequence
      of the Einstein flows $\{ {\mathcal E}^{(s_k)} \}_{k=1}^\infty$
      that converges to an element of ${\mathcal S}$
      in the pointed weak $W^{2,p}$-topology.
\end{definition}

We define ``approaches ${\mathcal S}$'' in the
pointed (norm) $C^{1,\alpha}$-topology similarly.
The motivation for these definitions comes from how one can
define convergence to a compact subset of a metric space, just
using the notion of sequential convergence. In our applications,
the relevant set ${\mathcal S}$ of Einstein flows can be taken to be
sequentially compact.

\begin{definition} \label{2.19}
We say that a pointed CMC
  Einstein flow
${\mathcal E}^1$ is
$\epsilon$-close to a  pointed CMC
  Einstein flow ${\mathcal E}^2$ in the pointed
$C^{1,\alpha}$-topology
if they are both defined on the time interval $(\epsilon, \epsilon^{-1})$ and, up to applying time-independent
pointed diffeomorphisms, 
the metrics are $\epsilon$-close in the $C^{1,\alpha}$-norm on $(\epsilon, \epsilon^{-1}) \times
B_{h_2(1)}(x^{(2)}, \epsilon^{-1})$.
\end{definition}

We don't make a similar definition of closeness for the
pointed weak $W^{2,p}$-topology because the weak topology is not
metrizable.

We now take $t = - \frac{n}{H}$, with $t$ ranging in an interval $(0, t_0]$.

\begin{definition} \label{old1.54}
  A type-I Einstein flow is a CMC Einstein flow
  for which there is some $C < \infty$ so that
  $|\Rm|_T \le C t^{-2}$ for all $t \in (0, t_0]$.
\end{definition}

Let $B_{h(t)}(x,t)$ denote the time-$t$ metric ball of radius $t$ around $x$.

\begin{definition} \label{2.56}
  If ${\mathcal E}$ is a CMC Einstein flow and $x \in X$ then the flow
  is noncollapsing at $x$ as $t \rightarrow 0$
    if $\vol \left( B_{h(t)}(x, t)  \right) \ge v_0 t^n$ for all $t$,
    for some $v_0 > 0$.
    \end{definition}

We do not know examples of crushing singularities for which the Einstein flow
fails to be everywhere noncollapsing.

\begin{proposition} \cite[Proposition 11]{Lott (2020)} \label{old1.55}
  Let ${\mathcal E}$ be a type-I
  Einstein flow on an $n$-dimensional manifold $X$.
  Suppose that it is
  defined on a time-interval $(0, t_0]$ and has complete time slices.
    Suppose that it is noncollapsing at $x \in X$ as $t \rightarrow 0$.
    Given a sequence $s_i \rightarrow 0$,
    after passing to a subsequence,
there is a limit
      $\lim_{i \rightarrow \infty} {\mathcal E}_{s_i} =
      {\mathcal E}^\infty$
      in the pointed weak $W^{2,p}$-topology and the pointed
      $C^{1,\alpha}$-topology. The limit Einstein flow
      ${\mathcal E}^\infty$ is
      defined on the time interval $(0, \infty)$. Its time slices
      $\{(X^\infty, h^\infty(u))\}_{u > 0}$ are complete. Its
      lapse function $L^\infty$ is uniformly bounded below by a
      positive constant.
\end{proposition}

\section{Asymptotic geometry} \label{sect3}

In this section we prove Theorems \ref{1.2}, \ref{1.4} and \ref{1.5}. We
initially prove them in the noncollapsing case.  In
Subsection \ref{subsect3.4} we indicate how to remove this assumption.

\subsection{Milne asymptotics} \label{subsect3.1}

Let ${\mathcal M}$ be the collection of pointed Einstein flows that
describe Lorentzian cones over pointed Riemannian Einstein $n$-manifolds with
Einstein constant $-(n-1)$. We take the basepoint for such a flow
to be at time one.
The proof of the next proposition is similar to that of
\cite[Proposition 3.5]{Lott (2018)}.

\begin{proposition} \label{3.1}
  Suppose that ${\mathcal E}$ is a type-I CMC Einstein flow.
  Fix $x \in X$. 
  Suppose that ${\mathcal E}$ is noncollapsing at $x$ and
  $\dvol_{h(t)}(x)$ is  $O(t^n)$ as $t \rightarrow 0$.
  Then as $s \rightarrow 0$, the rescaled Einstein flows ${\mathcal E}_s$,
  pointed at $x$, approach ${\mathcal M}$ in the pointed weak
  $W^{2,p}$-topology and the pointed $C^{1,\alpha}$-topology.
\end{proposition}
\begin{proof}
  Suppose that the claim is not true.  Let $\{s_i\}_{i=1}^\infty$ be a
  sequence with $\lim_{i \rightarrow \infty} s_i = 0$ such that
  no subsequence of $\{{\mathcal E}_{s_i}\}_{i=1}^\infty$ converges to an
  element of ${\mathcal M}$ in the given topologies.
From (\ref{old1.14}), we have $L \le 1$.
  From (\ref{old1.16}),
  $t^{-n} \dvol_{h(t)}(x)$ is monotonically nonincreasing in $t$, and
  \begin{equation}
    \log \frac{t^{-n} \dvol_{h(t)}(x)}{t_0^{-n} \dvol_{h(t_0)}(x)} =
    n \int_t^{t_0} (1-L(v,x)) \frac{dv}{v}.
    \end{equation}
  By assumption,   $t^{-n} \dvol_{h(t)}(x)$ is uniformly bounded above, so
  \begin{equation} \label{3.3}
    \int_0^{t_0} (1-L(v,x)) \frac{dv}{v} < \infty.
    \end{equation}

  From Proposition \ref{old1.55},
  after passing to a subsequence, we can assume that
  $\lim_{i \rightarrow \infty} {\mathcal E}_{s_i} = {\mathcal E}^\infty$
  for a pointed CMC Einstein flow ${\mathcal E}^\infty$.
  As the inequality $L \le 1$ passes to the limit, we
  know that $L^\infty \le 1$.
  We claim that
  $L^\infty(u,x^\infty) = 1$ for all $u \in (0, \infty)$.
  If not then $L^\infty(\widehat{u},x^\infty) \le 1 - \delta$ for some
  $\delta > 0$ and 
  $\widehat{u} \in (0, \infty)$.
  By continuity, there is some $\mu > 0$ so that
  $L^\infty(\widehat{u} e^\sigma, x^\infty) \le 1 - \frac{\delta}{2}$ for all
  $\sigma \in [ -\mu,\mu]$. 
  Then for sufficiently large $i$, we know that
  $L(s_i \widehat{u} e^\sigma, x) \le 1 - \frac{\delta}{4}$ for all
  $\sigma \in [-\mu, \mu]$.  After passing to a
  subsequence, we can assume that the intervals
  $\{[s_i \widehat{u} e^{-\mu},
  s_i \widehat{u} e^{\mu}]\}_{i=1}^\infty$ are disjoint. We obtain a
    contradiction to
  (\ref{3.3}).
 
  Thus $L^\infty(u,x^\infty) = 1$ for all $u \in (0, \infty)$. Equation
  (\ref{old1.13})
  (with $t$ replaced by $u$), along with elliptic regularity, the fact that
  $h^\infty$ is locally
$C^{1,\alpha}$-regular and the fact that $(K^\infty)^0$
is locally $C^\alpha$-regular, implies that $L^\infty(u, \cdot)$
is locally $C^{2,\alpha}$-regular.
We can apply the strong maximum principle to
(\ref{old1.13}) on $X^\infty$
to obtain that $L^\infty = 1$ and $(K^\infty)^0 = 0$.
As $K^\infty(u) = - \frac{1}{u} h^\infty(u)$, it follows from
(\ref{old1.4}) that $h^\infty(cu) = c^2 h^\infty(u)$.
Then (\ref{old1.5}) implies that
$h^\infty(1)$ is an Einstein manifold with Einstein constant $-(n-1)$.
Hence there is a
subsequence of $\{{\mathcal E}_{s_i}\}_{i=1}^\infty$ that converges to an
element of ${\mathcal M}$, which is a contradiction. 
\end{proof}

\subsection{Kasner-like asymptotics} \label{subsect3.2}

\begin{lemma} \label{3.4}
Given a function $\widehat{R}$ of $t$,
  if $R(t,x) \le \widehat{R}(t)$ for all $x \in X$ then
  $L(t,x) \ge \frac{n}{n^2+t^2 \widehat{R}(t)}$ for all $x \in X$,
  as long as the denominator is positive.
  \end{lemma}
  \begin{proof}
    This follows from applying the weak maximum principle to
    (\ref{old1.13}).
  \end{proof}

  Let ${\mathcal K}$ be the collection of pointed
  expanding CMC Einstein flows with
  $R = 0$, $L = \frac{1}{n}$ and $|K|^2 = H^2$, defined for $t \in
  (0, \infty)$. We take the basepoint for such a flow
to be at time one.
  
  \begin{proposition} \label{3.6}
    Let ${\mathcal E}$ be a type-I CMC Einstein flow
    with asymptotically nonpositive spatial scalar curvature.
    Suppose in addition that there is a
    nonnegative function $\widehat{R} : (0, t_0] \rightarrow
\R$
with $\int_0^{t_0} t^2 \widehat{R}(t) \frac{dt}{t} < \infty$ so that
$R(x,t) \le \widehat{R}(t)$ for all $x \in X$ and $t \in (0, t_0]$.
Fix $x \in X$.  Suppose that
    ${\mathcal E}$ is noncollapsing at $x$ and there is some $c > 0$
    so that 
    $\frac{1}{t} \dvol_{h(t)}(x) \ge \frac{c}{t_0} \dvol_{h(t_0)}$
    for all $t \in (0, t_0]$.
    Then as $s \rightarrow 0$, the rescaled Einstein flows
    ${\mathcal E}_s$, pointed at $x$, approach ${\mathcal K}$ in the pointed weak
    $W^{2,p}$-topology 
and the pointed $C^{1,\alpha}$-topology.
  \end{proposition}
  \begin{proof}
    Suppose that the claim fails.  Then there is a sequence
    $s_i \rightarrow 0$ with the property that no subsequence
    of ${\mathcal E}_{s_i}$
    converges to an element of ${\mathcal K}$ in the given topologies.

    After passing to a subsequence, we can assume that
    $\lim_{i \rightarrow \infty} {\mathcal E}_{s_i} = {\mathcal E}^\infty$
    for a CMC Einstein flow ${\mathcal E}^\infty$.
\begin{lemma} \label{upperL}
    $L^\infty(u, x^\infty) \le \frac{1}{n}$
  for all $u \in (0, \infty)$.
\end{lemma}
\begin{proof}
    Suppose that
    $L^\infty(u, x^\infty)
    \ge \frac{1}{n} + \delta$ for some $u \in (0, \infty)$
    and some $\delta > 0$. By continuity, there is
    some $\mu$ so that $L^\infty(u e^{\sigma}, x^\infty)
    \ge \frac{1}{n} + \frac12
    \delta$ for all $\sigma \in [- \mu, \mu]$. Then for large $i$,
    we know that
    $L(s_i u e^{\sigma}, x) \ge \frac{1}{n} + \frac14
    \delta$ for all $\sigma \in [- \mu, \mu]$. After passing to a
    subsequence, we can assume that the intervals
    $\{[s_i u e^{- \mu}, s_i u e^{\mu}]\}_{i=1}^\infty$ are disjoint.
    From Lemma \ref{3.4},
    \begin{equation}
      L(v,x) - \frac{1}{n} \ge \frac{n}{n^2+v^2 \widehat{R}(v)} - \frac{1}{n} =
      - \: \frac{v^2 \widehat{R}(v)}{n(n^2 + v^2 \widehat{R}(v))} \ge
      - \: \frac{v^2 \widehat{R}(v)}{n^3}.
    \end{equation}
    Hence the negative part of $L(v,x) - \frac{1}{n}$ is integrable
    with respect to $\frac{dv}{v}$. 
    It follows that
    \begin{equation}
      \int_0^{t_0} \left( L(v,x) - \frac{1}{n} \right) \frac{dv}{v} =
      \infty.
      \end{equation}
However, 
    from (\ref{2.20}), we have
    \begin{equation} \label{3.8}
      \log
      \frac
          {\frac{1}{t_0} \dvol_{h(t_0)}(x)}
        {\frac{1}{t} \dvol_{h(t)}(x)}
           = n \int_t^{t_0} \left( L(v,x) - \frac{1}{n} \right)
            \frac{dv}{v}.
    \end{equation}
    By our assumptions, the left-hand side of (\ref{3.8}) is bounded as
    $t \rightarrow 0$.
    This is a contradiction, so $L^\infty(u,x^\infty) \le \frac{1}{n}$ for all
    $u \in (0, \infty)$.
\end{proof}

\begin{lemma} \label{extra}
  $R^\infty \le 0$.
\end{lemma}
\begin{proof}
  Note that $R^\infty \in
  L^p_{loc}((0, \infty) \times X^\infty)$.
  Let $f$ be a compactly supported nonnegative
  continuous function on $(0, \infty) \times X^\infty$.
  Choose a compact interval $S \subset (0, \infty)$ and a compact
  $n$-dimensional manifold-with-boundary $W^\infty \subset X^\infty$
  so that the support of $f$ is contained in $S \times W^\infty$.
  With reference to Definition \ref{old1.49}, the weak $L^p$-convergence of
  scalar curvature gives
  \begin{equation}
    \int_S \int_{X^\infty} f R^\infty \dvol_{h^\infty(u)} u^2 du = \\
 \lim_{i \rightarrow \infty}
\int_S \int_{X^\infty} f \: (\Id \times \phi_{S,W^\infty, i})^* R_{s_i} \: 
\dvol_{h^\infty(u)} u^2 du.
  \end{equation}

  Since ${\mathcal E}$ has asymptotically nonpositive spatial scalar
  curvature, we can assume that
  $\lim_{t \rightarrow 0} t^2 \widehat{R}(t) = 0$.
  Now
  \begin{align}
& \int_S \int_{X^\infty} f \: (\Id \times \phi_{S,W^\infty, i})^* R_{s_i} \:
    \dvol_{h^\infty(u)} u^2 du \le \\
    &
\left( \int_S \int_{X^\infty} f \: 
\dvol_{h^\infty(u)} du \right) \max_{u \in S} s_i^2 u^2 \widehat{R}(s_i u).
\notag
\end{align}
Hence
  \begin{equation}
    \int_S \int_{X^\infty} f R^\infty \dvol_{h^\infty(u)} u^2 du \le 0
  \end{equation}
  for all such $f$, which proves the lemma.
\end{proof}
    
    \begin{lemma} \label{3.9}
      ${\mathcal E}^\infty \in {\mathcal K}$.
    \end{lemma}
    \begin{proof}
      From Lemma \ref{3.4} and the asymptotic nonpositive spatial scalar
      curvature, we know that
$L^\infty(u,x^\infty) \ge \frac{1}{n}$ for all
      $u \in (0, \infty)$.
      Then Lemma \ref{upperL} implies that
      $L^\infty(u,x^\infty) = \frac{1}{n}$  for all
      $u \in (0, \infty)$.. 
Fix $u$.
      By the same argument as in the proof of Proposition \ref{3.1},
      we know that
        $L^\infty(u,\cdot)$ is locally $C^{2,\alpha}$-regular.  Then we
        can apply the strong maximum principle to (\ref{old1.13}) to conclude
        that $L^\infty = \frac{1}{n}$. From (\ref{old1.13}) again, we obtain
        $R^\infty = 0$.
        (Note that from the constraint equation (\ref{old1.2}), $R^\infty$ is
        locally H\"older-continuous.) Then (\ref{old1.2}) gives
        $|K^\infty|^2 = H^2$. Thus
        ${\mathcal E}^\infty \in {\mathcal K}$.
      \end{proof}
    We have found a subsequence of $\{ {\mathcal E}_{s_i} \}_{i=1}^\infty$
    that converges to an element of ${\mathcal K}$, which is a contradiction.
This proves the proposition.
  \end{proof}

\subsection{Kasner-like time intervals}  

  \begin{proposition} \label{3.10}
    Suppose that ${\mathcal E}$ is a type-I CMC Einstein flow with asymptotically
    nonpositive spatial scalar curvature.  Fix $x \in X$.  Suppose that
    ${\mathcal E}$ is noncollapsing at $x$ and for each
    $\beta > 0$, 
    $\dvol_{h(t)}(x)$ fails to be $O(t^{1+\beta})$ as $t \rightarrow 0$.
    Given
    $\epsilon > 0$, let $S_\epsilon$ be the set of $\tau \ge 0$ so that
    ${\mathcal E}_{t_0 e^{- \tau}}$, pointed at $x$,
    is not $\epsilon$-close to ${\mathcal K}$ in the
    pointed $C^{1,\alpha}$-topology.
    Given $N \in \Z^+$, let $F(N)$ be the number of unit intervals
    $\{[k, k+1]\}_{k = 0}^{N-1}$ that have a nonempty intersection with
    $S_\epsilon$.
    Then
    \begin{equation}
      \liminf_{N \rightarrow \infty} \frac{F(N)}{N}
      = 0.
      \end{equation}
  \end{proposition}
  \begin{proof} We begin with a couple of lemmas.
    \begin{lemma} \label{3.12}
      Given $\epsilon > 0$, there is a $\delta > 0$ so that
      if $\tau \ge 2 \delta^{-1}$, and $L(t_0 e^{- \tau^\prime},x) \le
      \frac{1}{n} + \delta$ for $\tau^\prime \in [\tau - \delta^{-1},
        \tau + \delta^{-1}]$, then ${\mathcal E}_{t_0 e^{- \tau}}$
        is $\epsilon$-close to ${\mathcal K}$ in the pointed
        $C^{1, \alpha}$-topology.
      \end{lemma}
    \begin{proof}
      If the lemma is not true then there is a sequence
      $\delta_i \rightarrow 0$ and for each $i$,
      some $\tau_i \ge 2 \delta_i^{-1}$
      so that 
$L(t_0 e^{- \tau^\prime},x) \le
      \frac{1}{n} + \delta_i$ for $\tau^\prime \in [\tau_i - \delta_i^{-1},
        \tau_i + \delta_i^{-1}]$, but ${\mathcal E}_{t_0 e^{- \tau_i}}$
        is not $\epsilon$-close to ${\mathcal K}$ in the pointed
        $C^{1, \alpha}$-topology. After passing to a subsequence, we
        can assume that
        $\lim_{i \rightarrow \infty} {\mathcal E}_{t_0 e^{- \tau_i}} =
        {\mathcal E}^\infty$ in the pointed $C^{1,\alpha}$-topology, for some
        CMC Einstein flow ${\mathcal E}^\infty$ defined for time parameter
        $u \in (0, \infty)$. By construction,
        $L^\infty(u,x^\infty) \le \frac{1}{n}$ for all
        $u \in (0, \infty)$. As ${\mathcal E}$ has asymptotically nonpositive
        spatial scalar curvature, the proof of Lemma
        \ref{extra} shows that
        $R^\infty \le 0$. The proof of Lemma \ref{3.9} applies again, so
        ${\mathcal E}^\infty \in {\mathcal K}$. Hence
        ${\mathcal E}_{t_0 e^{- \tau_i}}$ is $\epsilon$-close to
        ${\mathcal K}$ for large $i$, which is a contradiction.
    \end{proof}

    \begin{lemma} \label{3.13}
      Given $\delta > 0$, there are $t^\prime, \mu > 0$
      with the following property.
      Suppose that $t \le t^\prime$ and $L(t, x) > \frac{1}{n} + \delta$.
      Then $L(t e^{\sigma}, x) > \frac{1}{n} + \frac12 \delta$
      for $\sigma \in [- \mu, \mu]$.
      \end{lemma}
    \begin{proof}
      If the claim is not true then
      taking $\mu_i = \frac{1}{i}$,
      there is a sequence $\{t_{i,j}\}_{j=1}^\infty$ with
      $\lim_{j \rightarrow \infty} t_{i,j} = 0$
      so that $L(t_{i,j}, x) > \frac{1}{n} + \delta$,
      but $L(t_{i,j} e^\sigma, x) \le \frac{1}{n} + \frac12 \delta$
      for some $\sigma \in [ - \mu_i, \mu_i ]$.
      Passing to a diagonal subsequence, there are sequences
      $\mu_i \rightarrow 0$ and $t_i \rightarrow 0$ so that
      $L(t_{i}, x) > \frac{1}{n} + \delta$
      but $L(t_{i} e^\sigma, x) \le \frac{1}{n} + \frac12 \delta$
      for some $\sigma \in [ - \mu_i, \mu_i ]$. Passing to a subsequence,
      we can assume that $\lim_{i \rightarrow \infty}
      {\mathcal E}_{t_i} = {\mathcal E}^\infty$ for some CMC
      Einstein flow ${\mathcal E}^\infty$. Then
      $L^\infty(1, x^\infty) \ge \frac{1}{n} + \delta$. By continuity, there
      is some $\mu^\prime > 0$ so that
      $L^\infty(e^{\sigma}, x^\infty) \ge \frac{1}{n} + \frac89 \delta$ for all
      $\sigma \in [- \mu^\prime, \mu^\prime]$. Then for large $i$,
      we know that
      $L(t_i e^{\sigma}, x) \ge \frac{1}{n} + \frac34 \delta$ for all
      $\sigma \in [- \mu^\prime, \mu^\prime]$. This is a contradiction.
      \end{proof}
    To prove the proposition, suppose that it is not true. Then
    there is some $c > 0$ so that 
    $F(N) \ge c N$ for large $N$.

Let $\delta$ be the parameter from Lemma \ref{3.12}.
For $N \in \Z^+$,
let $G(N)$ be the number of unit intervals $\{[k,k+1]\}_{k=0}^{N-1}$
    that contain a number $\sigma$ for which
    $L(t_0 e^{- \sigma}, x) > \frac{1}{n} + \delta$.
    Lemma \ref{3.12} implies that there is some $c^\prime > 0$ so that
    $G(N) \ge c^\prime N$ for large $N$.
    Lemmas \ref{3.4} and \ref{3.13}
    now imply that there is some $c^{\prime \prime} > 0$ so that
    \begin{equation}
      \int_{t}^{t_0} \left( L(v,x) - \frac{1}{n}
      \right) \frac{dv}{v} \ge c^{\prime \prime} \log \frac{t_0}{t}
    \end{equation}
    for small $t$. Equation (\ref{3.8}) then gives
      \begin{equation}
        \dvol_t(x) \le
        \left( \frac{t}{t_0} \right)^{1+nc^{\prime \prime}} \dvol_{t_0}(x)
      \end{equation}
      for small $t$,
      which contradicts the assumptions of the proposition.  This
      proves the proposition.
  \end{proof}
  
\subsection{Collapsing case} \label{subsect3.4}

We indicate how to remove the noncollapsing assumption in
Propositions \ref{3.1}, \ref{3.6} and \ref{3.10}.
Without this assumption, we can again take
pointed limits of Einstein flows but the limit flow will generally
be on an \'etale groupoid instead of a manifold.  For background
information on \'etale groupoids and Einstein flows on \'etale
groupoids, we refer to
\cite[Section 5]{Lott (2007)} and 
\cite[Section 3.1]{Lott (2018)}.
We can define $\epsilon$-closeness of Einstein flows on
\'etale groupoids in the $C^{1,\alpha}$-topology analogously to
Definition \ref{2.19}, 
using the setup of \cite[Definition 5.8]{Lott (2007)}.
The strong maximum principle applies on
the unit space of the groupoid directly.

We define ${\mathcal M}$ and ${\mathcal K}$ as in
Subsections \ref{subsect3.1} and
\ref{subsect3.2}, except
with the Einstein flows being on \'etale groupoids,
Then the proofs of Propositions \ref{3.1}, \ref{3.6} and
\ref{3.10} go through
without significant change.

\section{Examples} \label{sect4}

    \begin{example} \label{2.9}
        Consider a Lorentzian cone over a Riemannian Einstein $n$-manifold
        $(X,h_{Ein})$ with
        Einstein constant $- (n-1)$. The metric is
\begin{equation} \label{2.10}
  g = - dt^2 + t^2 h_{Ein}.
\end{equation}
The corresponding Einstein flow ${\mathcal E}$ is a type-I CMC Einstein flow.
It is noncollapsing at each $x \in X$ as $t \rightarrow 0$.
The volume density $\dvol_{h(t)}(x)$ is proportionate to $t^n$.
The pointed rescaling limit $\lim_{s \rightarrow 0} {\mathcal E}_s$ equals
${\mathcal E}$. This gives an example of Proposition \ref{3.1}.
    \end{example}

    \begin{example}
Consider the product of the previous example, in dimension $n - n^\prime$, with
  a flat torus $(T^{n^\prime}, h_{flat})$. The metric is
  \begin{equation} \label{2.11}
  g = - dt^2 + t^2 h_{Ein} + h_{flat}.
  \end{equation}
  The corresponding Einstein flow ${\mathcal E}$ is a type-I CMC Einstein flow
  and has nonpositive spatial scalar curvature.
It is noncollapsing at each $x \in X$ as $t \rightarrow 0$.
The volume density $\dvol_{h(t)}(x)$ is proportionate to $t^{n-n^\prime}$.

The pointed rescaling limit $\lim_{s \rightarrow 0} {\mathcal E}_s$ is the
product of the Einstein flow of the previous example with flat $\R^{n^\prime}$.
If $n - n^\prime = 1$ then we get an example of Proposition \ref{3.6}.
    \end{example}

    \begin{example} \label{4.5}
      Consider a Kasner solution on a flat
      $n$-manifold.  After possibly passing
  to a cover of $X$, the metric is
  \begin{equation} \label{2.13}
    g = - \: \frac{1}{n^2} dt^2 + (d \vec{x})^T t^{2M} d\vec{x}.
  \end{equation}
  Here $M$ is a symmetric $(n \times n)$-matrix with $\Tr(M) = \Tr(M^2) = 1$.
  We have written the metric so that $t \: = \: - \: \frac{n}{H}$. Then
  \begin{equation} \label{2.14}
    L = \frac{1}{n}, \: \: \: \: \: \:
    R = 0, \: \: \: \: \: \:
    |K|^2 = H^2 = \frac{n^2}{t^2}.
  \end{equation}
  The corresponding Einstein flow ${\mathcal E}$ is a type-I CMC Einstein flow
  and has vanishing spatial scalar curvature.
It is noncollapsing at each $x \in X$ as $t \rightarrow 0$.
The volume density $\dvol_{h(t)}(x)$ is proportionate to $t$.

The pointed rescaling limit $\lim_{s \rightarrow 0} {\mathcal E}_s$ is a
Kasner flow, with the same matrix $M$. It lives on $\R^n$ provided that
$M$ does not have $1$ as an eigenvalue.
This gives an example of Proposition \ref{3.6}.  
    \end{example}

    \begin{example} \label{4.8}
Consider a Kantowski-Sachs solution with $X$ diffeomorphic to
  $S^2 \times S^1$. The metric is a $\Z$-quotient of the interior of the
  event horizon in a Schwarzschild solution, after switching the
  usual $t$ and $r$ variables:
  \begin{equation} \label{4.9}
    g = - \:
    \frac{1}{\frac{2m}{t}-1} dt^2 + \left( \frac{2m}{t} - 1\right) dr^2
    + t^2 g_{S^2}.
  \end{equation}
  Here $t \in (0, 2m)$ and the $\Z$-quotienting is in the $r$-variable.
  The corresponding Einstein flow ${\mathcal E}$ is a type-I CMC Einstein flow,
  although the parameter $t$ in (\ref{4.9}) is not the Hubble time $t_H$.
  The relation is that for small time, $t \sim t_H^{\frac23}$.
  The spatial slices have scalar curvature $R(t) = \frac{2}{t^2}$, which
  goes like $t_H^{- \: \frac{4}{3}}$. It follows that in terms of the
    Hubble time,
    $\int_0^{t_0} v^2 R(v) \frac{dv}{v} < \infty$.
  The Einstein flow is noncollapsing at each $x \in X$ as $t_H \rightarrow 0$.
  The volume density $\dvol_{h(t_H)}(x)$ goes like $t_H$.
  This gives an example of Proposition \ref{3.6}.
The pointed rescaling limit $\lim_{s \rightarrow 0} {\mathcal E}_s$ is the
Kasner flow on $\R^3$ with a diagonal matrix $M$ whose diagonal entries are
$\left\{ \frac23, \frac23, - \frac13 \right\}$.
     \end{example}

    \begin{example} \label{4.10}
      Consider an Einstein flow on $S^3 \cong \SU(2)$
      that is left $\SU(2)$-invariant
      and right $\U(1)$-invariant. This is the Taub part of the
      Taub-NUT solution
      \cite[Section 9.2.7]{Ellis-Wainwright (1997)}.
      The corresponding Einstein flow is a type-I CMC Einstein flow.
      The spatial slices have positive scalar curvature but
      one can check that
    $\int_0^{t_0} v^2 R(v) \frac{dv}{v} < \infty$.      
  The Einstein flow is noncollapsing at each $x \in X$ as $t \rightarrow 0$.
  The volume density $\dvol_{h(t_H)}(x)$ goes like $t$.
  This gives an example of Proposition \ref{3.6}.
  Geometrically, before rescaling, the Gromov-Hausdorff limit
  as $t \rightarrow 0$ is a
  $2$-sphere and the circle fibers have length that goes like $t$.
  The pointed rescaling limit $\lim_{s \rightarrow 0} {\mathcal E}_s$ is a
  Kasner flow on $S^1 \times \R^2$ with a diagonal matrix $M$
  whose diagonal entries are
$\left\{ 1,0,0 \right\}$.
\end{example}

\begin{example} \label{4.11}
  Let ${\mathcal E}$ be a Bianchi-VIII NUT solution on a circle bundle
  over a higher genus surface
  \cite[Section 9.2.6]{Ellis-Wainwright (1997)}.
The corresponding Einstein flow is a type-I CMC Einstein flow.
The spatial slices have negative scalar curvature.
  The Einstein flow is noncollapsing at each $x \in X$ as $t \rightarrow 0$.
  The volume density $\dvol_{h(t_H)}(x)$ goes like $t$.
  This gives an example of Proposition \ref{3.6}.
  Geometrically, before rescaling, the Gromov-Hausdorff limit
  as $t \rightarrow 0$ is the surface
  and the circle fibers have length that goes like $t$.
  The pointed rescaling limit $\lim_{s \rightarrow 0} {\mathcal E}_s$ is a
  Kasner flow on $S^1 \times \R^2$ with a diagonal matrix $M$
  whose diagonal entries are
$\left\{ 1,0,0 \right\}$.
\end{example}    

\begin{example} \label{4.12}
  Consider a four dimensional polarized Gowdy spacetime with spatial slices
  diffeomorphic to $T^3$ \cite{Isenberg-Moncrief (1990)}.  The metric
  can be written
  \begin{equation} \label{4.13}
    g = e^{2a} (- dt^2 + d\theta^2) + t \left( e^W dx^2 + e^{-W} dy^2
    \right).  
    \end{equation}
  Here $a$ and $W$ are functions of $t$ and $\theta$.
  Define $\tau$ by $t = e^{- \tau}$, so one approaches the singularity as
  $\tau \rightarrow \infty$. 
  Asymptotics as $\tau \rightarrow \infty$ were derived in
  \cite{Isenberg-Moncrief (1990)}.
  In particular,
  \begin{align}
    a(\tau, \theta) & \sim \frac{1 - \pi^2(\theta)}{4} (\tau - \tau_0) +
    \alpha(\theta) + \ldots, \\
    W(\tau, \theta) & \sim \pi(\theta) (\tau - \tau_0) + \omega(\theta) +
    \ldots \notag
    \end{align}
  for appropriate functions $\alpha, \pi, \omega$ of $\theta$.

  As the singularity is a crushing
  singularity, there is a CMC foliation, but the metric 
  (\ref{4.13}) is not in CMC form. Because of this, one cannot read off
  whether the hypotheses of Proposition \ref{3.6} are fulfilled. Nevertheless,
  one can say the following.  First, from the barrier argument for
  CMC hypersurfaces \cite{Gerhardt (1983)}, one can show that the
  Hubble time $t_H$ goes like $e^{- \: \frac{\pi^2(\theta) + 3}{4} \tau}$.
  Then from \cite[Theorem IV.1]{Isenberg-Moncrief (1990)}, the Kretschmann
  scalar satisfies
\begin{equation}
  \left| R^{\alpha \beta \gamma \delta}
  R_{\alpha \beta \gamma \delta} \right| \le \const t_H^{-4},
\end{equation}
which is consistent with type-I asymptotics. Next, we can estimate the
spacetime volume $V(t_H)$
between the singularity and the CMC slice with
$H = - \: \frac{3}{t_H}$. One finds that $V(t_H)$ goes like $t_H^2$,
which is consistent with $\dvol_{h(t_H)}$ going like $t_H$.

Fixing $\theta$, $x$ and $y$, the geometry as $\tau \rightarrow \infty$
approaches a Kasner geometry on $\R^3$ with a diagonal matrix $M$
whose diagonal entries are
$\left\{ \frac{\pi^2(\theta) - 1}{\pi^2(\theta) + 3},
\frac{2 - 2 \pi(\theta)}{\pi^2(\theta) + 3},
\frac{2 + 2 \pi(\theta)}{\pi^2(\theta) + 3} \right\}$.
\end{example}

\begin{example} \label{2.55}
  Consider an Einstein flow of Bianchi type IX
\cite[Section 6.4]{Ellis-Wainwright (1997)}. It is a 
homogeneous CMC Einstein flow on $S^3$ with a crushing singularity.
The Einstein flow is left $\SU(2)$-invariant.  The case when it is
right $\U(1)$-invariant was already considered in Example \ref{4.10}, so
we assume that the flow is not right $\U(1)$-invariant, i.e. it is of
Mixmaster type.

\begin{proposition} \label{4.17}
  A Bianchi IX Mixmaster flow is type-I and has asymptotically nonpositive
  spatial scalar curvature.
  \end{proposition}
\begin{proof}
  To describe the ODE of the Bianchi IX flow, we use the normalizations of
  \cite{Ringstrom (2001)}; see \cite[Appendix]{Ringstrom (2001)}.
  The metric can be written as
  \begin{equation}
    g = - dt^2 + h(t).
  \end{equation}
  (The $t$ here is not the Hubble time.)
Put $\theta = - H$.
A new dimensionless variable
$\tau$ (not related to the $\tau$ of Proposition \ref{3.10}) is
    defined by
    \begin{equation}
      \frac{dt}{d\tau} = \frac{3}{\theta}.
    \end{equation}
    We normalize so that
      $\tau = 0$ corresponds to $t = t_0$.
    Then
    \begin{equation}
g = - \left( \frac{3}{\theta} \right)^2 d\tau^2 + h(\tau).      
    \end{equation}
    That is, the lapse function is
\begin{equation} \label{lapse}
  L(\tau) = \frac{3}{\theta}.
  \end{equation}
    Approaching the singularity corresponds to $\tau \rightarrow - \infty$.
    
As the spacetime Ricci tensor
  vanishes, the curvature tensor is determined by the spacetime
  Weyl curvature.
  Since $\dim(X) = 3$, the Weyl curvature is expressed in terms of
  ``electric'' and ''magnetic''
  tensors \cite[Section 1.1.3]{Ellis-Wainwright (1997)}.
  After normalization by the Hubble time, the tensor components
  can be written as
  polynomials in the Wainwright-Hsu variables
  $\Sigma_+, \Sigma_-, N_+, N_-, N_1$ \cite[(6.37)]{Ellis-Wainwright (1997)}.
    Hence the Einstein flow will be type-I provided that these
    variables remain bounded as one approaches the singularity.
    From
    \cite{Ringstrom (2001)}, this is the case.

    The normalized spatial scalar curvature is 
\begin{equation}
  \frac{R}{H^2} = - \: \frac12
  \left[ N_1^2 + N_2^2 + N_3^2 - 2 (N_1 N_2 + N_2 N_3 + N_3 N_1)
  \right].
\end{equation}
Going toward the singularity,
the flow approaches an attractor where two of the $N_i$'s vanish
\cite{Ringstrom (2001)}.  Hence
the Einstein flow has asymptotically nonpositive spatial scalar curvature.
\end{proof}

We now restrict to a certain class of Mixmaster flows.
The Kasner circle is $\{(\Sigma_+, \Sigma_-) \: : \:
\Sigma_+^2 + \Sigma_-^2 = 1\}$.
The Kasner map is a certain degree two map of the Kasner circle to
itself \cite[Section 6.4.1]{Ellis-Wainwright (1997)}.
Given a periodic orbit of the Kasner map that is not a fixed point, there
is a heteroclinic cycle of the ODE that consists of the periodic points on
the Kasner circle, joined by Taub type-II Einstein
flows that asymptotically approach
two adjacent points in the orbit as time goes to $\pm \infty$.
There is a family of Mixmaster flows that asymptotically
approach the heteroclinic cycle as $\tau \rightarrow - \infty$
\cite{Liebscher-Harterich-Webster-Georgi (2011)}.

\begin{proposition} \label{4.23}
These Mixmaster flows satisfy the assumptions of Proposition \ref{3.10}.  
  \end{proposition}
\begin{proof}
  We first show that $\dvol_{h(t_H)}(x)$ fails to be
  $O \left( t_H^{1+\beta} \right)$ for any $\beta > 0$, where the
  $t_H$ denotes the Hubble time. The variable $\tau$ is defined in a way
  that there is a simple dependence of the volume form on $\tau$.  Namely,
\begin{equation} \label{4.24}
  \dvol_{h(\tau)}(x) = e^{3 \tau} \dvol_{h(0)}(x).
\end{equation}
To see this, equation (\ref{old1.4}) gives
\begin{equation}
  \frac{d}{d\tau} \dvol_{h(\tau)}(x) = -L H \dvol_{h(\tau)}(x).
  \end{equation}
Then (\ref{lapse}) implies
\begin{equation}
  \frac{d}{d\tau} \dvol_{h(\tau)}(x) = 3 \dvol_{h(\tau)}(x),
\end{equation}
from which
  (\ref{4.24}) follows.

  Using (\ref{4.24}), to see how $\dvol_{h(\tau)}(x)$ depends on the
  Hubble time $t_H$, since $t_H = \frac{3}{\theta}$
  it suffices to see how $\theta$ depends on $\tau$.
  One has
    \begin{equation} \label{4.21}
      \frac{d\theta}{d\tau} =
      - \left( 1 + 2 \Sigma_+^2 + 2 \Sigma_-^2 \right) \theta,
      \end{equation}
so
  \begin{equation}
    \log \theta(\tau) - \log \theta(0) = -
    3 \tau + 2 \int_{\tau}^0 \left( \Sigma_+^2 + \Sigma_-^2 - 1
    \right)(u) \: du.
        \end{equation}
  From \cite[Sections 3 and 4]{Liebscher-Harterich-Webster-Georgi (2011)}, as
  $\tau \rightarrow - \infty$, the trajectory in the $(\Sigma_+, \Sigma_-)$-plane will spend almost all of its time near the periodic orbit on the Kasner
  circle. Consequently, $\int_{\tau}^0 \left( \Sigma_+^2 + \Sigma_-^2 - 1
  \right)(u) \: du$ will be sublinear in $|\tau|$ as
    $\tau \rightarrow - \infty$. Thus to leading order,
    $\theta(\tau)$ will go like $e^{- 3 \tau}$.
    Then for any $\beta > 0$, it follows that $\dvol_{h(\tau)}(x)$
    will fail to be $O \left( t_H^{1+\beta} \right)$ as
    $\tau \rightarrow - \infty$, i.e. as $t_H \rightarrow 0$.

    We now show that the Einstein flow is noncollapsing at $x$ as
    $t_H \rightarrow 0$.
    Because the Einstein flow is type-I, the rescaling
    $t_H^{-2} h(t_H)$ of the metric at Hubble time $t_H$ 
    has a double sided curvature bound that is independent of $t_H$.
    Suppose that the Einstein flow is collapsing at $x$, as witnessed by 
    a sequence
    of times $\{t_H^j\}_{j=1}^\infty$ going to zero. 
    From the theory of bounded curvature collapse
    \cite{Cheeger-Fukaya-Gromov (1992)}, there is some constant $a > 0$
    so that for large $j$, there is a loop at $x$ that is
    homotopically nontrivial in the metric ball
    $B_{(t_H^j)^{-2} h(t_H^j)}(x,a)$, with the
    length of the loop (with respect to $(t_H^j)^{-2} h(t_H^j)$) going to
    zero as $j \rightarrow \infty$.

    On the other hand, the evolution of $h$ is given by
    (\ref{old1.4}).  Let $\{\sigma^i\}_{i=1}^3$ be the coframe to the
    orthonormal frame $\{e_i\}_{i=1}^3$ used in deriving the ODE,
    so
    \begin{equation}
      h = \sigma^1 \otimes \sigma^1 + \sigma^2 \otimes \sigma^2 +
      \sigma^3 \otimes \sigma^3.
    \end{equation}
    Then
    \begin{align}
      K = & \frac{H}{3} (1 - 2 \Sigma_+) \sigma^1 \otimes \sigma^1 +
      H \left( \frac13 + \frac13 \Sigma_+ + \frac{1}{\sqrt{3}} \Sigma_- \right)
      \sigma^2 \otimes \sigma^2 + \\      
      &
      H \left( \frac13 + \frac13 \Sigma_+ - \frac{1}{\sqrt{3}} \Sigma_- \right)
      \sigma^3 \otimes \sigma^3. \notag
    \end{align}
    Using (\ref{old1.4}) and (\ref{4.21}), one obtains
    \begin{align}
      \frac{d(H^2 h)}{d\tau} = &
      (-4\Sigma_+ - 4 \Sigma_+^2 - 4 \Sigma_-^2) H^2 \sigma^1 \otimes \sigma^1 + \\
      & (2\Sigma_+ + 2 \sqrt{3} \Sigma_- - 4 \Sigma_+^2 - 4 \Sigma_-^2) H^2 \sigma^2 \otimes \sigma^2 + \notag \\
      & (2\Sigma_+ - 2 \sqrt{3} \Sigma_- - 4 \Sigma_+^2 - 4 \Sigma_-^2) H^2 \sigma^2 \otimes \sigma^2. \notag
    \end{align}
    The region
in the $(\Sigma_+, \Sigma_-)$-plane
    where $-4\Sigma_+ - 4 \Sigma_+^2 - 4 \Sigma_-^2 \ge 0$ is
    the closed disk with center $\left( - \: \frac12, 0 \right)$ and radius
    $\frac12$. The regions where the other coefficients are nonnegative
    are the rotations of this disk around the origin by $\frac{2 \pi}{3}$
    and $\frac{4 \pi}{3}$ radians.  In particular, these three disks only
    meet the Kasner circle at the Taub points $\left\{
    e^{\frac{i \pi}{3}},e^{i \pi}, e^{\frac{5 \pi}{3}} \right\}$, which
    correspond to flat Kasner spacetimes.  Since the heteroclinic cycle
    avoids these points, as $\tau \rightarrow - \infty$
    the normalized metric $\theta^2 h$ will be
    greatly expanded during the time spent
    near the Kasner circle and will have
    bounded contraction the rest of the time,
    again using
    \cite[Sections 3 and 4]{Liebscher-Harterich-Webster-Georgi (2011)}.
    Hence there is a large overall expansion and so at
    the fixed time $\tau = 0$,
    there is a sequence of loops at $x$ that are homotopically nontrivial
    and whose lengths go to zero. This is a contradiction. 
\end{proof}

Because of the homogeneity, in this case
the Kasner-like regions in the conclusion
of Proposition \ref{3.10} are standard Kasner geometries in the sense of
Example \ref{4.5}.
It seems plausible that the hypotheses of Proposition \ref{3.10}
are also satisfied for the Mixmaster
spacetimes considered in
\cite{Beguin (2010),Brehm (2016)}.
\end{example}

\begin{example} \label{2.79}
  Consider a locally homogeneous Einstein flow of Bianchi type VIII.
  The spatial geometry is a quotient of
  $\widetilde{\SL(2, \R)}$. The spatial scalar curvature is
  nonpositive.

  The lifted geometry on $\widetilde{\SL(2, \R)}$ is
  left $\widetilde{\SL(2, \R)}$-invariant.  The case when it is
  right $\widetilde{\SO(2)}$-invariant was essentially considered in
  Example \ref{4.11}, so we assume that the Einstein flow is not
  right $\widetilde{\SO(2)}$-invariant.

  For a generic set of initial conditions, a
    Bianchi VIII solution converges as $\tau \rightarrow - \infty$ to
   the Mixmaster attractor \cite{Brehm (2016)}.  
   Then Proposition \ref{4.17} extends to these solutions.
   As mentioned in
   \cite{Liebscher-Harterich-Webster-Georgi (2011)}, the results of
   that paper extend to the construction of Bianchi VIII solutions that,
   as $\tau \rightarrow - \infty$, approach a
   heteroclinic cycle coming from a periodic orbit of the Kasner map.
   (These solutions are in the generic set of
   \cite{Brehm (2016)}.) Then Proposition \ref{4.23} extends to such
   solutions.
\end{example}

\end{document}